\documentclass[a4paper,11pt,leqno]{article}
\usepackage{bbm}
\usepackage{amssymb, amsmath, amscd, amsfonts, amsthm}
\usepackage[body={13cm, 21cm}, centering, dvipdfm]{geometry}
\usepackage[tiny]{titlesec}
\usepackage[symbol]{footmisc}
\usepackage{mathrsfs}
\usepackage{enumerate, varioref}
\usepackage{graphicx}
\usepackage[square, numbers, sort&compress]{natbib}
\usepackage{marvosym}
\usepackage[colorlinks,linkcolor=red,anchorcolor=blue,citecolor=green,hyperindex,dvipdfm]{hyperref}
\DefineFNsymbols{number}{1}\setfnsymbol{number}

 \newtheorem{theorem}{Theorem}[section]
 \newtheorem{corollary}[theorem]{Corollary}
 \newtheorem{lemma}[theorem]{Lemma}
 
 \theoremstyle{remark}
 \newtheorem{remark}[theorem]{Remark}
 \newtheorem{example}[theorem]{Example}
 
 \numberwithin{equation}{section}
\title{\textbf{\large Entry-exit decisions with implementation delay under uncertainty}}

\author{\small Yong-Chao ZHANG\footnote{School of Mathematics and Statistics, Northeastern University at Qinhuangdao, Taishan Road 143, Qinhuangdao 066004, China. E-mail: zhangyc@neuq.edu.cn.}}
\date{}

\begin{document}

\maketitle
\vspace{-0.5cm}
\begin{abstract}
We employ a natural method from the perspective of the optimal stopping theory to analyze entry-exit decisions with implementation delay of a project, and provide closed expressions for optimal entry decision times, optimal exit decision times and the maximal expected present value from the project. The results in conventional research were obtained under the restriction that the sum of the entry cost and the exit cost is nonnegative. In practice, we usually meet this sum is negative, so it is necessary to remove the restriction. If the sum is negative, there may exist two price triggers of entry decision, which does not happen when the sum is nonnegative, and it is not optimal to enter and then immediately exit the project even though it is an arbitrage opportunity.\\
\noindent\textit{Key Words}: entry decision time; exit decision time; implementation delay; optimal stopping problem; viscosity solution\\
\noindent\textit{Mathematics Subject Classification (2010)}: 60G40; 91B06\\
\end{abstract}

\section{Introduction}\label{intro}

The background to entry-exit decisions is described as follows \cite{Zhang2015}. A firm has an option to invest in a project as well as stop it. To start the project activity, the firm needs an initial investment cost to produce a commodity at a running cost. In addition, it may stop the project at a terminal investment cost.

What time is optimal to decide to enter the project and what time is optimal to decide to exit the project? Many authors answered these two questions in the setting that there is no time lag between decision times and corresponding implementation times. For example, see \cite{Dix89,Shi97,SucZer2000,LumlZerv2001,Wang2005,Sodal2006,Tsek2010,Leven2005,BoyLev2007, Zhang2015}.

In practice, a major characteristic of investments is that there exist lags between decision times and corresponding implementation times. Some authors discussed entry-exit decision problems with implementation delay. For example, see \cite{BaSt1996, GautMore2000, Ok2005,CoScTa2008}.

In \cite{BaSt1996}, Bar-Ilan and Strange embedded lags in the classic model presented by Dixit \cite{Dix89}. They considered entry and exit decisions by employing the real option theory and derived a system of equations (see Equations (22)--(25) in \cite{BaSt1996}), then obtained semi-closed solutions for entry and exit decisions. However, they did not proved the existence and uniqueness of the solution to the system. Gauthier and Morellec \cite{GautMore2000} provided more explicit solutions through assuming \textit{a priori} the forms of decision times. In \cite{Ok2005}, {\O}ksendal studied two optimal exit decision problems with implementation delay---an assets selling problem and a resource extraction problem. In \cite{CoScTa2008}, Costeniuc \textit{et al.}~applied the probabilistic approach to entry and exit decisions with Parisian implementation delay from the view of real options.

The results in conventional research were obtained under the assumption that the sum of the entry cost and the exit cost is nonnegative. In practice, we usually meet this sum is negative. For example, an investor buys a project at a low price and then sells it at a high price. We will remove this assumption and study the case where the sum is negative.

If the sum is nonnegative, there exists no arbitrage opportunity, and there is only one price trigger of entry decision. However, if the sum is negative, it is an arbitrage opportunity to enter and then immediately exit the project, and there may be two price triggers of entry decision (\eqref{i:DoubEntrTrig} of Theorem \ref{t:EntrDeci}). We find that it is not optimal to enter and then immediately exit the project even if the sum is negative (see \eqref{i:Inst2} and \eqref{i:DoubEntrTrig} of Theorem \ref{t:EntrDeci}).

In this paper, we employ a method from the perspective of the optimal stopping theory, which proves to be natural, to rigorously discuss the entry-exit decision problem with implementation delay. We study this problem by three steps. First, we transform the delayed case into an instant case. Second, we decompose the instant case into two standard optimal stopping problems, and then solve these two problems. Finally, we provide explicitly an optimal entry decision time, an optimal exit decision time and an expression of the maximal expected present value from the project.

We outline the structure of this paper. In Section \ref{OptStop}, we recall briefly the classical optimal stopping theory. In Section \ref{Trans}, we show that delayed optimal stopping problems involving two stopping times can be transformed to instant ones. In Section \ref{Model}, we describe the model in detail. In Section \ref{EnExDe},  we obtain an optimal entry-exit decision as to when the firm decides to enter the project and when the firm decides to exit the project (Theorem \ref{t:optimalIOD}).

\section{Some results concerning classical optimal stopping problems}\label{OptStop}
In this section, we recall briefly some results of classical optimal stopping problems. For details, we refer to \cite[section 5.2]{Pham2009}.

Let $(\Omega, \mathscr{F}, \{\mathscr{F}_t\}_{t\geq 0}, \mathbb{P})$ be a filtered probability space with $\{\mathscr{F}_t\}_{t\geq 0}$ satisfying the usual conditions and $\mathscr{F}_0$ being the completion of $\{\emptyset, \Omega\}$. Let $B=(B(t), t\geq 0)$ be a $d$-dimensional standard Brownian motion defined on $(\Omega, \mathscr{F}, \{\mathscr{F}_t\}_{t\geq 0}, \mathbb{P})$.

Let $X=(X(t), t\geq 0)$ be a diffusion in $\mathbb{R}^n$ given by
\[
\mathrm{d}X(t)=\alpha(X(t)) \mathrm{d}t+\beta(X(t)) \mathrm{d}B(t), \;\; X(0)=x,
\]
where $\alpha: \mathbb{R}^n\rightarrow\mathbb{R}^n$ and $\beta: \mathbb{R}^n\rightarrow\mathbb{R}^{n\times
d}$ are some Lipschitz functions.

Let $\mathcal{T}$ denote the set of all stopping times valued in $[0,+\infty]$.

\begin{theorem}\label{t:ClasOpti}
Consider the following optimal stopping problem
\begin{equation}\label{e:ClasOpti}
V(x):=\sup_{\tau\in\mathcal{T}}\mathbb{E}^x\left[\int_0^{\tau}\exp(-rt)f(X(t))\mathrm{d}t+\exp(-r\tau)g(X(\tau))\right]
\end{equation}
for some Lipschitz functions $f$ and $g$. Here $\mathbb{E}^x[~\cdot~]:=\mathbb{E}[~\cdot~|X(0)=x]$, and $\exp(-r\tau)g(X(\tau))\equiv 0$ on $\{\tau=+\infty\}$.

Assume that $r>0$ is large enough. Then the following are true.
\begin{enumerate}[{\rm (i)}]

\item\label{i:ViscSolu} The value function $V$ is Lipschitz continuous and is the unique viscosity solution with linear growth of the variational inequality
\[
\min\{rV-\mathcal{L}V-f,V-g\}=0,
\]
where $\mathcal{L}$ is the infinitesimal generator of $X$.

\item\label{i:Maxi} Set $\mathcal{S}:=\{x:x\in\mathbb{R}^n, V(x)=g(x)\}$ which is called the exercise region. Then $\tau^*:=\inf\{t: t>0, X(t)\in \mathcal{S}\}$ is a maximizer of the problem \eqref{e:ClasOpti}.

\item\label{i:C2Cont} The value function $V$ is a viscosity solution of
\[
rV-\mathcal{L}V-f=0\qquad on\;\; \mathcal{C},
\]
where $\mathcal{C}:=\{x: x\in\mathbb{R}^n, V(x)>g(x)\}$ is the continuation region; moreover, if $\mathcal{L}$ is locally uniformly elliptic, $V$ is $C^2$ on $\mathcal{C}$.

\item\label{i:C1Cont} Assume that $X$ is 1-dimensional, $\mathcal{L}$ is locally uniformly elliptic, and $g$ is $C^1$ on $\mathcal{S}$. Then $V$ is $C^1$ on $\partial\mathcal{C}$ and $C^2$ at the isolated points of $\mathcal{S}$.

\item\label{i:StopImme} Define a function $\widehat{V}$ by
\[
\widehat{V}(x):=\mathbb{E}\left[\int_0^{+\infty}\exp(-rt)f(X(t))\mathrm{d}t\right].
\]
Then $\mathcal{S}=\emptyset$ implies $\widehat{V}\geq g$ and $\widehat{V}\geq g$ implies $V=\widehat{V}$.

\item\label{i:ExerSub} If $g$ is $C^2$ continuous on some open set $\mathcal{O}$, then $\mathcal{S}\subset \{x: x\in\mathcal{O}, rg(x)-\mathcal{L}g(x)-f(x)\geq 0\}\cup\mathcal{O}^\mathrm{c}$.

\item\label{i:ExerForm} Assume that $X$ is 1-dimensional and takes values in $(0,+\infty)$, $X(t, x)\rightarrow X(t, 0)=0$ as $x\rightarrow 0$, $\widehat{V}(x_0)<g(x_0)$ for some $x_0>0$, and $g$ is $C^2$ continuous. We have the following two facts. If $\mathcal{D}=[a,+\infty)$ for some $a>0$, where $\mathcal{D}:=\{x:x>0, rg(x)-\mathcal{L}g(x)-f(x)\geq 0\}$, then $\mathcal{S}=[x^*,+\infty)$ for some $x^*\in[a,+\infty)$. If $g(0)\geq f(0)/r$ and $\mathcal{D}=(0,a]$ for some $a>0$, then $\mathcal{S}=(0,x^*]$ for some $x^*\in(0,a]$.
\end{enumerate}
\end{theorem}
\begin{proof}
We refer to \cite[section 5.2]{Pham2009} for the proof.
\end{proof}

\section{A useful transformation}\label{Trans}
In this section, we show that delayed optimal stopping problems involving two stopping times can be transformed to instant ones. The proof is similar to that of \cite[p.~38, Theorem 2.11]{OkSu07}.

\begin{theorem}\label{t:Transf}
Let $\delta$ be a nonnegative number. Consider the following two optimal stopping problems
\begin{equation}\label{e:GenDel}
J(x):=\sup_{\tau_1,\tau_2\in\mathcal{T},\atop\tau_1\leq\tau_2}\mathbb{E}^x\left[\int_{\tau_1+\delta}^{\tau_2+\delta}f(X(t))\mathrm{d}t
+g_1(X(\tau_1+\delta))+g_2(X(\tau_2+\delta))\right],
\end{equation}
where $f,g_1,g_2: \mathbb{R}^n\rightarrow\mathbb{R}$ are three functions such that the expectations are finite;
\begin{equation}\label{e:GenIns}
\widetilde{J}(x):=\sup_{\tau_1,\tau_2\in\mathcal{T},\atop\tau_1\leq\tau_2}\mathbb{E}^x\left[\int_{\tau_1}^{\tau_2}f(X(t))\mathrm{d}t
+g_1^{\delta}(X(\tau_1))+g_2^{\delta}(X(\tau_2))\right],
\end{equation}
where
\[
g_1^{\delta}(x):=\mathbb{E}^x\left[-\int_{0}^{\delta}f(X(t))\mathrm{d}t+g_1(X(\delta))\right],
\]
and
\[
g_2^{\delta}(x):=\mathbb{E}^x\left[\int_{0}^{\delta}f(X(t))\mathrm{d}t+g_2(X(\delta))\right].
\]
Then $J(x)=\widetilde{J}(x)$. In addition, if $(\tau_1^*,\tau_2^*)$ is a maximizer of \eqref{e:GenIns}, it is also a maximizer of \eqref{e:GenDel}.
\end{theorem}
\begin{proof}
1. Note that
\[
\begin{split}
\mathbb{E}^x&\left[\int_{\tau_1+\delta}^{\tau_2+\delta}f(X(t))\mathrm{d}t
+g_1(X(\tau_1+\delta))+g_2(X(\tau_2+\delta))\right]\\
&=\mathbb{E}^x\left[\left(\int_{\tau_1}^{\tau_2}-\int_{\tau_1}^{\tau_1+\delta}+\int_{\tau_2}^{\tau_2+\delta}\right)f(X(t))\mathrm{d}t\right.\\
&\qquad\qquad\qquad\left.+g_1(X(\tau_1+\delta))+g_2(X(\tau_2+\delta))\right]\\
&=\mathbb{E}^x\left[\int_{\tau_1}^{\tau_2}f(X(t))\mathrm{d}t
-\int_{\tau_1}^{\tau_1+\delta}f(X(t))\mathrm{d}t+g_1(X(\tau_1+\delta))\right.\\
&\qquad\qquad\qquad\qquad\qquad\left.+\int_{\tau_2}^{\tau_2+\delta}f(X(t))\mathrm{d}t+g_2(X(\tau_2+\delta))\right].
\end{split}
\]
Then, by the strong Markov property of the process $X$, we get
\[
\begin{split}
\mathbb{E}^x&\left[\int_{\tau_1+\delta}^{\tau_2+\delta}f(X(t))\mathrm{d}t
+g_1(X(\tau_1+\delta))+g_2(X(\tau_2+\delta))\right]\\
&=\mathbb{E}^x\left[\int_{\tau_1}^{\tau_2}f(X(t))\mathrm{d}t
+\mathbb{E}^{X(\tau_1)}\left[-\int_{0}^{\delta}f(X(t))\mathrm{d}t+g_1(X(\delta))\right]\right.\\
&\qquad\qquad\qquad\qquad\qquad\left.+\mathbb{E}^{X(\tau_2)}\left[\int_{0}^{\delta}f(X(t))\mathrm{d}t+g_2(X(\delta))\right]\right]\\
&=\mathbb{E}^x\left[\int_{\tau_1}^{\tau_2}f(X(t))\mathrm{d}t
+g_1^{\delta}(X(\tau_1))+g_2^{\delta}(X(\tau_2))\right],
\end{split}
\]
which completes the proof.
\end{proof}

\section{The model}\label{Model}
We return to the entry-exit decision problem introduced in Section \ref{intro}, and assume that the price process $P$ follows
\begin{equation}\label{e:price1D}
\text{$\mathrm{d}P(t)=\mu P(t)\mathrm{d}t+\sigma P(t)\mathrm{d}B(t)$ and $P(0)=p$,}
\end{equation}
where $\mu\in \mathbb{R}$, $\sigma,\,p>0$, and $B$ is a one dimensional standard Brownian motion, which denotes uncertainty.

Applying It\^{o}'s formula, we deduce that the solution of the equation (\ref{e:price1D}) is
\begin{equation}\label{e:price2D}
P(t)=P(0)\exp\left[\left(\mu-\frac{1}{2}\sigma^2\right)t+\sigma B(t)\right].
\end{equation}

To answer the two questions---what time is optimal to make an entry decision and what time is optimal to make an exit decision, we will solve the following optimal problem
\begin{equation}\label{e:optimalD}
\begin{aligned}
J(p):=\sup\limits_{\tau_I\leq\tau_O}\mathbb{E}^p\left[\int_{\tau_I+\delta}^{\tau_O+\delta}\exp(-r t)(P(t)-C)\mathrm{d}t \right.&-\exp(-r (\tau_I+\delta))K_I\\
&\left.-\exp(-r(\tau_O+\delta))K_O\right],
\end{aligned}
\end{equation}
where $\tau_I$ and $\tau_O$ are stopping times, $r$ is the discount rate such that $r>0$, $C$ is the running cost, $K_I$ is the entry cost, $K_O$ is the exit cost, and the nonnegative number $\delta$ is a time lag between the decision time and the corresponding implementation time. We call stopping times $\tau_I$ and $\tau_O$ an entry decision time and an exit decision time, respectively, and the function $J$ the maximal expected present value from the project.

\begin{remark}
\begin{enumerate}[(1)]
\item We do not propose any restriction on the running cost, entry cost and exit cost, except for constants.
\item Note that for any stopping time $\tau$ and nonnegative number $\delta$, $\tau+\delta$ is also a stopping time. The maximal expected present value $J$ of the delayed case is no more than that of the corresponding instant case. We may interpret their difference as the loss due to delayed implementation.
\item Furthermore, let $0\leq\delta_1<\delta_2<+\infty$ and $\mathcal{T}$ is the collection of all stopping times, then $\{\tau+\delta_2: \tau\in\mathcal{T}\}\subset\{\tau+\delta_1: \tau\in\mathcal{T}\}$, thus the value of $J$ corresponding to $\delta_2$ is no more than that corresponding to $\delta_1$. This implies the following principle: {\em once one has made a right decision, he/she should activate it as soon as possible}.
\end{enumerate}
\end{remark}

\section{An optimal entry-exit decision}\label{EnExDe}
In this section, we provide an optimal entry-exit decision and an explicit expression for the function $J$.

Let us first consider a simple case $r\leq\mu$. In this case, noting the expression \eqref{e:price2D} of $P$, we have
\begin{equation*}
\begin{aligned}
\mathbb{E}^p&\left[\int_{\delta}^{+\infty}\exp(-r t)(P(t)-C)\mathrm{d}t\right]\\
&=\int_{\delta}^{+\infty}\exp(-rt)\left(p\exp(\mu t)-C\right)\mathrm{d}t\\
&=\left\{
\begin{aligned}
&\lim\limits_{t\rightarrow+\infty}\left(p\,(t-\delta)+\frac{C}{r}\left(\exp(-rt)-\exp(-r\delta)\right)\right),\;\;\text{if}\; r=\mu,\\
&\lim\limits_{t\rightarrow+\infty}\left(\frac{p}{\mu-r}\left(\exp((\mu-r)t)-\exp((\mu-r)\delta)\right)\right.\\
&\qquad\qquad\left.+\frac{C}{r}\left(\exp(-rt)-\exp(-r\delta)\right)\right),\;\;\text{if}\; r<\mu
\end{aligned}\right.\\
&=+\infty,
\end{aligned}
\end{equation*}
where we have used the fact that the process
\[
\left(\exp\left(-\frac{1}{2}\sigma^2 t+\sigma B(t)\right), t\geq 0\right)
\] is a martingale (see \cite[p.~288, Corollary 5.2.2]{App2009}) for the first step.

Thus we obtain the following result.
\begin{theorem}\label{t:optimal1D}
Assume that $r\leq \mu$. Then $\tau_I^*:=0$ a.s.\ is an optimal entry decision time and $\tau_O^*:=+\infty$ is an optimal exit decision time, i.e., the firm should never exit the project. In addition, the function $J$ in \eqref{e:optimalD} is given by $J\equiv+\infty$.
\end{theorem}

Now we determine an optimal entry-exit decision for the case $r>\mu$. To this end, we first employ Theorem \ref{t:Transf} to transform the delayed optimal stopping problem \eqref{e:optimalD} to an instant one.
\begin{theorem}\label{t:TranSpec}
The delayed optimal stopping problem \eqref{e:optimalD} is equivalent to the following optimal stopping problem
\begin{equation}\label{e:optimalIns}
\begin{split}
\widetilde{J}(p):=\sup\limits_{\tau_I\leq\tau_O}\mathbb{E}^p&\left[\int_{\tau_I}^{\tau_O}\exp(-r t)(P(t)-C)\mathrm{d}t\right.\\
&\quad\left.-\exp(-r\tau_I)(k_1P(\tau_I)+k_0)-\exp(-r \tau_O)(l_1P(\tau_O)+l_0)\right],
\end{split}
\end{equation}
where
\[
k_1:=\frac{\exp((\mu-r)\delta)-1}{\mu-r},\quad k_0:=\frac{C}{r}(\exp(-r\delta)-1)+\exp(-r\delta)K_I,
\]
\[
l_1:=-\frac{\exp((\mu-r)\delta)-1}{\mu-r},\quad l_0:=-\frac{C}{r}(\exp(-r\delta)-1)+\exp(-r\delta)K_O.
\]
\end{theorem}
\begin{proof}
1. Define the process $X$ by $X(t):=\left[s+t, P(t)\right]^\text{T}$, where $s\in\mathbb{R}$. Then
\[
\mathrm{d}X(t)=\left[
\begin{array}{c}
1\\
\mu P(t)
\end{array}\right]\mathrm{d}t+\left[
\begin{array}{c}
0\\
\sigma P(t)
\end{array}\right]\mathrm{d}B(t),\;\,X(0)=\left[
\begin{array}{c}
s\\
p
\end{array}\right].
\]

\noindent 2. According to Theorem \ref{t:Transf}, we need to calculate
\begin{equation}\label{e:gIdel}
\mathbb{E}^p\left[-\int_{0}^{\delta}\exp(-r(s+t))\left(P(t)-C\right)\mathrm{d}t-\exp(-r(s+\delta))K_I\right]
\end{equation}
and
\begin{equation}\label{e:gOdel}
\mathbb{E}^p\left[\int_{0}^{\delta}\exp(-r(s+t))\left(P(t)-C\right)\mathrm{d}t-\exp(-r(s+\delta))K_O\right].
\end{equation}

For \eqref{e:gIdel}, we have
\[
\begin{aligned}
\mathbb{E}^p&\left[-\int_{0}^{\delta}\exp(-r(s+t))\left(P(t)-C\right)\mathrm{d}t-\exp(-r(s+\delta))K_I\right]\\
&=-\int_0^{\delta}\exp(-r(s+t))\left(p\exp(\mu t)-C\right)\mathrm{d}t-\exp(-r(s+\delta))K_I\\
&=-\exp(-rs)\left(\frac{p}{\mu-r}\left(\exp((\mu-r)\delta)-1\right)\right.\\
&\qquad\qquad\qquad\qquad\left.+\frac{C}{r}\left(\exp(-r\delta)-1\right)+\exp(-r\delta)K_I\right),
\end{aligned}
\]
where we have used the fact that the process
\[
\left(\exp\left(-\frac{1}{2}\sigma^2 t+\sigma B(t)\right), t\geq 0\right)
\] is a martingale (see \cite[p.~288, Corollary 5.2.2]{App2009}) for the first step.

Similarly, we can calculate \eqref{e:gOdel}. Therefore, in light of Theorem \ref{t:Transf}, the delayed optimal stopping problem \eqref{e:optimalD} is equivalent to the optimal stopping problem \eqref{e:optimalIns}.
\end{proof}

In order to solve the optimal stopping problem \eqref{e:optimalIns}, we will solve the following two the optimal stopping problems
\begin{equation}\label{e:ExitDeci}
G(p):=\sup\limits_{\tau_O}\mathbb{E}^p\left[\int_{0}^{\tau_O}\exp(-r t)(P(t)-C)\mathrm{d}t-\exp(-r\tau_O)(l_1P(\tau_O)+l_0)\right]
\end{equation}
and
\begin{equation}\label{e:EntrDeci}
H(p):=\sup\limits_{\tau_I}\mathbb{E}^p\left[\exp(-r\tau_I)\left(G(P(\tau_I))-k_1P(\tau_I)-k_0\right)\right].
\end{equation}

Assume that $r>\mu$. Let $\lambda_1$ and $\lambda_2$ be the solutions of the quadratic equation
\[
r-\mu\lambda-\frac{1}{2}\sigma^2\lambda(\lambda-1)=0
\]
with $\lambda_1<\lambda_2$. Then we have $\lambda_1<0$ and $\lambda_2>1$.

\begin{theorem}\label{t:ExitTrigger}
For the optimal stopping problem \eqref{e:ExitDeci}, the following are true.
\begin{enumerate}[{\rm (i)}]
\item If $r>\mu$ and $C\leq rK_O$, then $\tau_O^*:=+\infty$ a.s.\ is a maximizer of \eqref{e:ExitDeci}. In addition, $G(p)=p/(r-\mu)-C/r$.

\item If $r>\mu$ and $C>rK_O$, then $\tau_O^*:=\inf\{t: t>0, P(t)\leq p_O\}$ a.s.\ is a maximizer of \eqref{e:ExitDeci}, where
\[
p_O=\exp(-\mu\delta)\frac{\lambda_1}{\lambda_1-1}(r-\mu)\left(\frac{C}{r}-K_O\right).
\]
In addition,
\[
G(p)=\left\{
\begin{aligned}
&A p^{\lambda_1}+\frac{p}{r-\mu}-\frac{C}{r},\;\;\text{if}\;\, p>p_O,\\
&\frac{p}{\mu-r}\left(\exp((\mu-r)\delta)-1\right)\\
&\qquad\quad+\frac{C}{r}\left(\exp(-r\delta)-1\right)-\exp(-r\delta)K_O,\;\;\text{if}\;\, p\leq p_O,
\end{aligned}
\right.
\]
where $A=\exp((\mu-r) \delta){p_O}^{1-\lambda_1}/(\lambda_1(\mu-r))$.
\end{enumerate}
\end{theorem}
\begin{proof}
1. Assume that $r>\mu$ and $C\leq rK_O$.

Noting that
\[
\mathbb{E}^p\left[\int_0^{+\infty}\exp(-rt)(P(t)-C)\right]=\frac{p}{r-\mu}-\frac{C}{r},
\]
we have
\[
\mathbb{E}^p\left[\int_0^{+\infty}\exp(-rt)(P(t)-C)\right]\geq -l_1p-l_0.
\]

Therefore, by \eqref{i:StopImme} of Theorem \ref{t:ClasOpti}, we achieve (1).

\noindent 2. Assume that $r>\mu$ and $C>rK_O$.

In this case, we have $\mathcal{D}=(0,\exp(-\mu\delta)(C-rK_O)]$. Thus, by \eqref{i:ExerForm} of Theorem \ref{t:ClasOpti}, the exercise region is of the form $(0, p_O]$ for some $p_O\in(0,+\infty)$. On the continuation region $(p_O,+\infty)$, $G$ satisfies the equation
\[
rG-\mu p G'-\frac{1}{2}\sigma^2p^2G''-p+C=0
\]
by \eqref{i:C2Cont} of Theorem \ref{t:ClasOpti}. Furthermore, by the Lipschitz property of $G$, we have
\[
G(p)=Ap^{\lambda_1}+\frac{p}{r-\mu}-\frac{C}{r}
\]
for some constant $A$.

Note that $G$ is $C^1$ continuous at $p_O$ by \eqref{i:C1Cont} of Theorem \ref{t:ClasOpti}. We get the following system
\[\left\{
\begin{aligned}
&Ap_O^{\lambda_1}+\frac{p_O}{r-\mu}-\frac{C}{r}=\frac{p_O}{\mu-r}\left(\exp((\mu-r)\delta)-1\right)\\
&\qquad\qquad\qquad\qquad\qquad\qquad+\frac{C}{r}\left(\exp(-r\delta)-1\right)-\exp(-r\delta)K_O\\
&\lambda_1A p_O^{\lambda_1-1}+\frac{1}{r-\mu}=\frac{\exp((\mu-r)\delta)-1}{\mu-r},
\end{aligned}\right.
\]
from which we obtain
\begin{equation}\label{e:pO}
p_O=\exp(-\mu \delta)\frac{\lambda_1}{\lambda_1-1}(r-\mu)\left(\frac{C}{r}-K_O\right)
\end{equation}
and
\[
A=\exp((\mu-r) \delta)\frac{{p_O}^{1-\lambda_1}}{\lambda_1(\mu-r)}.
\]
The proof is complete.
\end{proof}

\begin{remark}
We will prove in Theorem \ref{t:optimalIOD} that $p_O$ is the price trigger of exit decision.
\end{remark}

\begin{corollary}\label{c:ExTrUpBo}
The optimal exit trigger price $p_O$ in Theorem {\rm\ref{t:ExitTrigger}} satisfies $p_O<\exp(-\mu\delta)\left(C-rK_O\right)$.
\end{corollary}
\begin{proof}
Note that $1/\lambda_1<\mu/r$. Then, thanks to \eqref{e:pO}, the conclusion follows.
\end{proof}

\begin{theorem}\label{t:EntrDeci}
For the optimal stopping problem \eqref{e:EntrDeci}, the following are true.
\begin{enumerate}[{\rm (i)}]
\item If $r>\mu$, $C-rK_O\leq 0$ and $C+rK_I\leq 0$, then $\tau_I^*:=0$ a.s.\ is a maximizer of \eqref{e:EntrDeci}. In addition,
\[
H(p)=\frac{\exp((\mu-r)\delta)}{r-\mu}p-\exp(-r\delta)\left(\frac{C}{r}+K_I\right).
\]

\item\label{i:EntrTrig1} If $r>\mu$, $C-rK_O\leq 0$ and $C+rK_I> 0$, then $\tau_I^*:=\inf\{t: t>0, P(t)\geq p_I\}$ a.s.\ is a maximizer of \eqref{e:EntrDeci}, where
\[
p_I=\exp(-\mu \delta)\frac{\lambda_2}{\lambda_2-1}(r-\mu)\left(\frac{C}{r}+K_I\right).
\]
In addition,
\[
H(p)=\left\{
\begin{aligned}
&B p^{\lambda_2},\;\;\text{if}\;\, p<p_I,\\
&\frac{\exp((\mu-r)\delta)}{r-\mu}p-\exp(-r\delta)\left(\frac{C}{r}+K_I\right),\;\;\text{if}\;\, p\geq p_I,
\end{aligned}
\right.
\]
where $B=\exp((\mu-r) \delta){p_I}^{1-\lambda_2}/(\lambda_2(r-\mu))$.

\item\label{i:Inst2} If $r>\mu$, $C-rK_O>0$ and $C+rK_I\leq 0$, then $\tau_I^*:=0$ a.s.\ is a maximizer of \eqref{e:EntrDeci}.
In addition,
\[
H(p)=\left\{
\begin{aligned}
&A p^{\lambda_1}+\frac{\exp((\mu-r)\delta)}{r-\mu}p-\exp(-r\delta)\left(\frac{C}{r}+K_I\right),\;\;\text{if}\;\, p>p_O,\\
&-\exp(-r\delta)(K_I+K_O),\;\;\text{if}\;\, p\leq p_O.
\end{aligned}
\right.
\]

\item\label{i:EntrTrig2} If $r>\mu$, $C-rK_O>0$, $C+rK_I> 0$ and $K_I+K_O\geq0$, then $\tau_I^*:=\inf\{t: t>0, P(t)\geq p_I\}$ a.s.\ is a maximizer of \eqref{e:EntrDeci}, where $p_I$ is the largest solution of the algebraic equation
\[
A(\lambda_2-\lambda_1)p_I^{\lambda_1}+\frac{\exp((\mu-r)\delta)}{r-\mu}(\lambda_2-1)p_I-\lambda_2\exp(-r\delta)\left(\frac{C}{r}+K_I\right)=0.
\]
In addition,
\[
H(p)=\left\{
\begin{aligned}
&B p^{\lambda_2},\;\;\text{if}\;\, p<p_I,\\
&Ap^{\lambda_1}+\frac{\exp((\mu-r)\delta)}{r-\mu}p-\exp(-r\delta)\left(\frac{C}{r}+K_I\right),\;\;\text{if}\;\, p\geq p_I,
\end{aligned}
\right.
\]
where
\[
B=\lambda_1\lambda_2^{-1}Ap_I^{\lambda_1-\lambda_2}+\exp((\mu-r) \delta)\frac{{p_I}^{1-\lambda_2}}{\lambda_2(r-\mu)}.
\]

\item If $r>\mu$, $C-rK_O> 0$, $C+rK_I>0$, $K_I+K_O<0$ and $p_O\geq\exp(-\mu\delta)(C+rK_I)$, then $\tau_I^*:=0$ a.s.\ is a maximizer of \eqref{e:EntrDeci}. In addition,
\[
H(p)=\left\{
\begin{aligned}
&A p^{\lambda_1}+\frac{\exp((\mu-r)\delta)}{r-\mu}p-\exp(-r\delta)\left(\frac{C}{r}+K_I\right),\;\;\text{if}\;\, p>p_O,\\
&-\exp(-r\delta)(K_I+K_O),\;\;\text{if}\;\, p\leq p_O.
\end{aligned}
\right.
\]

\item\label{i:DoubEntrTrig} If $r>\mu$, $C-rK_O> 0$, $C+rK_I>0$, $K_I+K_O<0$ and $p_O<\exp(-\mu\delta)(C+rK_I)$, then $\tau_I^*:=\inf\{t: t>0, P(t)\leq p_I^{(1)}\ \text{or}\ P(t)\geq p_I^{(2)}\}$ a.s.\ is a maximizer of \eqref{e:EntrDeci}, where $(p_I^{(1)}, p_I^{(2)})$ is the solution of the equation
\begin{equation}\label{e:pI1pI2}
\begin{split}
&\left[
\begin{array}{cc}
\lambda_2{p_I^{(1)}}^{-\lambda_1}&-{p_I^{(1)}}^{-\lambda_1}\\
-\lambda_1{p_I^{(1)}}^{-\lambda_2}&{p_I^{(1)}}^{-\lambda_2}
\end{array}
\right]
\left[
\begin{array}{c}
-\exp(-r\delta)(K_I+K_O)\\
0
\end{array}
\right]\\
&\quad=\left[
\begin{array}{cc}
\lambda_2{p_I^{(2)}}^{-\lambda_1}&-{p_I^{(2)}}^{-\lambda_1}\\
-\lambda_1{p_I^{(2)}}^{-\lambda_2}&{p_I^{(2)}}^{-\lambda_2}
\end{array}
\right]
\left[
\begin{array}{c}
\begin{split}A{p_I^{(2)}}^{\lambda_1}&+\frac{\exp((\mu-r)\delta)}{r-\mu}p_I^{(2)}\\
&-\exp(-r\delta)\left(\frac{C}{r}+K_I\right)\end{split}\\
\lambda_1A{p_I^{(2)}}^{\lambda_1}+\frac{\exp((\mu-r)\delta)}{r-\mu}p_I^{(2)}
\end{array}
\right]
\end{split}
\end{equation}
with $p_I^{(1)}<p_I^{(2)}$.

In addition,
\[
H(p)=\left\{
\begin{aligned}
&A p^{\lambda_1}+\frac{\exp((\mu-r)\delta)}{r-\mu}p-\exp(-r\delta)\left(\frac{C}{r}+K_I\right),\;\;\text{if}\;\, p\geq p_I^{(2)},\\
&B_1p^{\lambda_1}+B_2p^{\lambda_2},\;\;\text{if}\;\, p_I^{(1)}< p< p_I^{(2)},\\
&-\exp(-r\delta)(K_I+K_O),\;\;\text{if}\;\, p\leq p_I^{(1)},
\end{aligned}
\right.
\]
where
\[
B_1=-\frac{\lambda_2{p_I^{(1)}}^{-\lambda_1}\exp(-r\delta)}{\lambda_2-\lambda_1}(K_I+K_O)
\]
and
\[
B_2=\frac{\lambda_1{p_I^{(1)}}^{-\lambda_2}\exp(-r\delta)}{\lambda_2-\lambda_1}(K_I+K_O).
\]
\end{enumerate}
\end{theorem}
\begin{proof}
1. Assume that $r>\mu$, $C-rK_O\leq 0$ and $C+rK_I\leq 0$.

Define a function $w$ by
\[
w(p):=\frac{\exp((\mu-r)\delta)}{r-\mu}p-\exp(-r\delta)\left(\frac{C}{r}+K_I\right)\qquad {\rm for}\;p\in(0,+\infty).
\]
Then we have
\[
rw(p)-\mu p w'(p)-\frac{1}{2}\sigma^2p^2w''(p)\geq 0,
\]
which implies $w$ is a viscosity solution of
\[
\min\{rV-\mu p V'-\frac{1}{2}\sigma^2p^2V'',V-w\}=0\qquad{\rm on}\; (0,+\infty).
\]

Note that $H(0^+)=w(0^+)$. Thus, by uniqueness of viscosity solutions (see \eqref{i:ViscSolu} of Theorem \ref{t:ClasOpti}), we have $H(p)=w(p)$. Consequently, the exercise region is of the $(0,+\infty)$, i.e., $\tau_I^*:=0$ a.s.\ is a maximizer of \eqref{e:EntrDeci} by \eqref{i:Maxi} of Theorem \ref{t:ClasOpti}.

\noindent 2. Assume that $r>\mu$, $C-rK_O\leq 0$ and $C+rK_I> 0$.

In this case, we have $\mathcal{D}=[\exp(-\mu\delta)(C+rK_I),+\infty)$. Thus, by \eqref{i:ExerForm} of Theorem \ref{t:ClasOpti}, the exercise region is of the form $[p_I,+\infty)$ for some $p_I\in(0,\infty)$. On the continuation region $(0,p_I)$, $H$ satisfies the equation
\[
rH-\mu p H'-\frac{1}{2}\sigma^2p^2H''=0
\]
by (7) of Theorem \ref{t:ClasOpti}. Furthermore, by the Lipschitz property of $H$, we have $H(p)=Bp^{\lambda_2}$ for some constant $B$.

Note that $H$ is $C^1$ continuous at $p_I$ by \eqref{i:C1Cont} of Theorem \ref{t:ClasOpti}. We get the following system
\[\left\{
\begin{aligned}
&Bp_I^{\lambda_2}=\frac{\exp((\mu-r)\delta)}{r-\mu}p_I-\exp(-r\delta)\left(\frac{C}{r}+K_I\right)\\
&\lambda_2B p_I^{\lambda_2-1}=\frac{\exp((\mu-r)\delta)}{r-\mu},
\end{aligned}\right.
\]
from which we obtain
\[
p_I=\exp(-\mu \delta)\frac{\lambda_2}{\lambda_2-1}(r-\mu)\left(\frac{C}{r}+K_I\right)
\]
and
\[
B=\exp((\mu-r) \delta)\frac{{p_I}^{1-\lambda_2}}{\lambda_2(r-\mu)}.
\]

\noindent 3. Assume that $r>\mu$, $C-rK_O>0$ and $C+rK_I\leq 0$.

Define a function $w$ by
\[
w(p):=\left\{
\begin{aligned}
&A p^{\lambda_1}+\frac{\exp((\mu-r)\delta)}{r-\mu}p-\exp(-r\delta)\left(\frac{C}{r}+K_I\right),\;\;\text{if}\;\, p>p_O,\\
&-\exp(-r\delta)(K_I+K_O),\;\;\text{if}\;\, p\leq p_O.
\end{aligned}
\right.
\]
Then $w$ is a viscosity subsolution of
\begin{equation}\label{e:SubSolu1}
\min\{rV-\mu p V'-\frac{1}{2}\sigma^2p^2V'',V-w\}=0\qquad{\rm on}\; (0,+\infty).
\end{equation}

We prove that $w$ is a viscosity supersolution of \eqref{e:SubSolu1}. To this end, we only need to prove
\begin{equation}\label{e:SupeSolu1}
rw(p_O)-\mu p \varphi'(p_O)-\frac{1}{2}\sigma^2p^2\varphi''(p_O)\geq 0
\end{equation}
for some function $\varphi\in C^2(\mathcal{N}(p_O))$ such that $w(p_O)=\varphi(p_O)$ and $w(p)\geq\varphi(p)$ on some neighbourhood $\mathcal{N}(p_O)$ of $p_O$, since $rw-\mu p w'-\sigma^2p^2w''/2\geq0$ on $(0,p_O)\cup(p_O+\infty)$.

Noting that $p_O$ is a minimizer of $w-\varphi$ on $\mathcal{N}(p_O)$, we have $w'(p_O)-\varphi'(p_O)=0$ and $w_-''(p_O)-\varphi''(p_O)\geq 0$, i.e., $\varphi'(p_O)=0$ and $\varphi''(p_O)\leq 0$. In addition, thanks to $C-rK_O>0$ and $C+rK_I\leq 0$, $K_I+K_O<0$. So \eqref{e:SupeSolu1} holds.

In summery, $w$ is a viscosity solution of
\[
\min\{rV-\mu p V'-\frac{1}{2}\sigma^2p^2V'',V-w\}=0\qquad{\rm on}\; (0,+\infty).
\]

Note that $H(0^+)=w(0^+)$. Then, by uniqueness of viscosity solutions (see \eqref{i:ViscSolu} of Theorem \ref{t:ClasOpti}), we get $H(p)=w(p)$ for $p\in(0,+\infty)$. Consequently, the exercise region is $(0,+\infty)$, i.e., $\tau_I^*:=0$ a.s.\ is a maximizer of \eqref{e:EntrDeci} by \eqref{i:Maxi} of Theorem \ref{t:ClasOpti}.

\noindent 4. Assume that $r>\mu$, $C-rK_O>0$, $C+rK_I> 0$ and $K_I+K_O\geq0$.

First assume $K_I+K_O>0$. Then, in light of \eqref{i:ExerSub} of Theorem \ref{t:ClasOpti}, we have
\[
\mathcal{S}\subset[\exp(-\mu\delta)(C+rK_I),+\infty)\cup\{p_O\}.
\]
Note that $p_O<\exp(-\mu\delta)(C-rK_O)$ by Corollary \ref{c:ExTrUpBo}, and $K_I+K_O>0$. These imply $p_O<\exp(-\mu\delta)(C+rK_I)$. Consequently, by following the proof of \eqref{i:ExerForm} of Theorem \ref{t:ClasOpti} (see \cite[p.~104]{Pham2009}) and using \eqref{i:C1Cont} of Theorem \ref{t:ClasOpti}, there is a point $p_I\in[\exp(-\mu\delta)(C+rK_I),+\infty)$ such that
\[
H(p)=G(p)-k_1p-k_0\qquad\text{for $p\in[p_I,+\infty)$}
\]
and
\begin{equation}\label{e:ContRegi1}
rH-\mu p H'-\frac{1}{2}\sigma^2p^2H''=0\;\qquad\text{on $(0,p_I)$}.
\end{equation}

Thus, by the Lipschitz property of $H$, we have $H(p)=Bp^{\lambda_2}$ for some constant $B$ from \eqref{e:ContRegi1}.

Note that $H$ is $C^1$ continuous at $p_I$ by \eqref{i:C1Cont} of Theorem \ref{t:ClasOpti}. We get the following system
\begin{equation}\label{e:C1Cont}
\left\{
\begin{aligned}
&Bp_I^{\lambda_2}=Ap_I^{\lambda_1}+\frac{\exp((\mu-r)\delta)}{r-\mu}p_I-\exp(-r\delta)\left(\frac{C}{r}+K_I\right)\\
&\lambda_2B p_I^{\lambda_2-1}=\lambda_1Ap_I^{\lambda_1-1}+\frac{\exp((\mu-r)\delta)}{r-\mu},
\end{aligned}\right.
\end{equation}
from which we obtain
\[
A(\lambda_2-\lambda_1)p_I^{\lambda_1}+\frac{\exp((\mu-r)\delta)}{r-\mu}(\lambda_2-1)p_I-\lambda_2\exp(-r\delta)\left(\frac{C}{r}+K_I\right)=0.
\]

We will show ahead in Lemma \ref{l:LargSolu} that the above algebraic equation has only two roots. One is less than $p_O$, and the other is greater than $p_O$. Since $p_I\geq\exp(-\mu\delta)(C+rK_I)$, $p_O<\exp(-\mu\delta)(C-rK_O)$ by Corollary \ref{c:ExTrUpBo}, and $K_I+K_O>0$, we must choose the greater one. Furthermore, we have
\[
B=\lambda_1\lambda_2^{-1}Ap_I^{\lambda_1-\lambda_2}+\exp((\mu-r) \delta)\frac{{p_I}^{1-\lambda_2}}{\lambda_2(r-\mu)}.
\]

For proving the exercise region is $[p_I,+\infty)$, we only need to show
\begin{equation}\label{e:ElimpO}
H(p_O)>G(p_O)-k_1p_O-k_0.
\end{equation}
To see this, consider the function $f(p):=H(p)-G(p)+k_1p+k_0$ for $p\in [0, p_O]$. Then we have $f(0)=\exp(-r\delta)(K_I+K_O)>0$. In addition, $f'(p)=\lambda_2 Bp^{\lambda_2-1}>0$ for $p\in (0, p_O)$, since $B>0$ by \eqref{e:C1Cont} and Lemma \ref{l:LargSolu}. The inequality \eqref{e:ElimpO} follows.

Now consider the case $K_I+K_O=0$. We refer to the following Step 6. To solve the systems \eqref{e:C1ContLeft} and \eqref{e:C1ContRigh}, we put $B_1=p_I^{(1)}=0$, then the systems \eqref{e:C1ContLeft} and \eqref{e:C1ContRigh} are reduced to \eqref{e:C1Cont}. By repeating the proof of the case $K_I+K_O>0$, we achieve our aim.

\noindent 5. Assume that $r>\mu$, $C-rK_O> 0$, $C+rK_I>0$, $K_I+K_O< 0$ and $p_O\geq\exp(-\mu\delta)(C+rK_I)$. The proof of this case is the same as that of the case \eqref{i:Inst2}.

\noindent 6. Assume that $r>\mu$, $C-rK_O> 0$, $C+rK_I>0$, $K_I+K_O< 0$ and $p_O<\exp(-\mu\delta)(C+rK_I)$.

In this case, we have
\[
\mathcal{S}\subset(0,p_O]\cup[\exp(-\mu\delta)(C+rK_I),+\infty).
\]
Thus, by following the proof of \eqref{i:ExerForm} of Theorem \ref{t:ClasOpti} (see \cite[p.~104]{Pham2009}), the exercise region is of the form $(0,p_I^{(1)}]\cup[p_I^{(2)},+\infty)$ for some $p_I^{(1)}\in(0,p_O]$ and $p_I^{(2)}\in[\exp(-\mu\delta)(C+rK_I),+\infty)$. On the continuation region $(p_I^{(1)},p_I^{(2)})$, $H$ satisfies the equation
\[
rH-\mu p H'-\frac{1}{2}\sigma^2p^2H''=0
\]
by \eqref{i:C2Cont} of Theorem \ref{t:ClasOpti}.

Thus we have $H(p)=B_1p^{\lambda_1}+B_2p^{\lambda_2}$ on $(p_I^{(1)},p_I^{(2)})$ for some constant $B_1$ and $B_2$.

Note that $H$ is $C^1$ continuous at $p_I^{(1)}$ and $p_I^{(2)}$ by \eqref{i:C1Cont} of Theorem \ref{t:ClasOpti}. We get the following systems
\begin{equation}\label{e:C1ContLeft}
\left\{
\begin{aligned}
&B_1{p_I^{(1)}}^{\lambda_1}+B_2{p_I^{(1)}}^{\lambda_2}=-\exp(-r\delta)(K_I+K_O)\\
&\lambda_1B_1{p_I^{(1)}}^{\lambda_1-1}+\lambda_2B_2{p_I^{(1)}}^{\lambda_2-1}=0
\end{aligned}\right.
\end{equation}
and
\begin{equation}\label{e:C1ContRigh}
\left\{
\begin{aligned}
&B_1{p_I^{(2)}}^{\lambda_1}+B_2{p_I^{(2)}}^{\lambda_2}=A{p_I^{(2)}}^{\lambda_1}+\frac{\exp((\mu-r)\delta)}{r-\mu}p_I^{(2)}-\exp(-r\delta)\left(\frac{C}{r}+K_I\right)\\
&\lambda_1B_1{p_I^{(2)}}^{\lambda_1-1}+\lambda_2B_2{p_I^{(2)}}^{\lambda_2-1}=\lambda_1A{p_I^{(2)}}^{\lambda_1-1}+\frac{\exp((\mu-r)\delta)}{r-\mu},
\end{aligned}\right.
\end{equation}
from which, by solving $B_1$ and $B_2$, respectively, we obtain \eqref{e:pI1pI2}.
\end{proof}

\begin{remark}
We will prove in Theorem \ref{t:optimalIOD} that $p_I$, $p_I^{(1)}$ and $p_I^{(2)}$ are the price triggers of entry decision.
\end{remark}

\begin{corollary}
The optimal entry triggers $p_I$'s in \eqref{i:EntrTrig1} and \eqref{i:EntrTrig2} of Theorem {\rm\ref{t:EntrDeci}} satisfy $p_I>\exp(-\mu\delta)(C+rK_I)$; the optimal entry triggers $p_I^{(1)}$ and $p_I^{(2)}$ in \eqref{i:DoubEntrTrig} of Theorem {\rm\ref{t:EntrDeci}} satisfy $p_I^{(1)}<\exp(-\mu\delta)(C-rK_O)$ and $p_I^{(2)}>\exp(-\mu\delta)(C+rK_I)$, respectively.
\end{corollary}
\begin{proof}
1. For the case \eqref{i:EntrTrig1} of Theorem \ref{t:EntrDeci}, we have
\[
p_I=\exp(-\mu \delta)\frac{\lambda_2}{\lambda_2-1}(r-\mu)\left(\frac{C}{r}+K_I\right).
\]
In addition, note that $1/\lambda_2>\mu/r$. The inequality $p_I>\exp(-\mu\delta)(C+rK_I)$ follows.

\noindent 2. Consider the case \eqref{i:EntrTrig2} of Theorem \ref{t:EntrDeci}.

Define a function $U$ by
\[
U(p):=Bp^{\lambda_2}-G(p)+k_1p+k_0\qquad\text{for $p\in[p_O,+\infty)$}.
\]
Then we have $U(p_O)\geq0$ and $U(p_I)=0$.

We prove the equation $U''(p)=0$ has a solution in $(p_O,p_I)$. To this end, suppose that the equation $U''(p)=0$ has no solution in $(p_O,p_I)$. Then the function $U'(\cdot)$ is strictly monotonous on $[p_O,p_I]$. In addition, note that $U'(p_O)=\lambda_2Bp_O^{\lambda_2-1}>0$ and $U'(p_I)=0$. We get $U'(p)>0$ for $p_O<p<p_I$. Consequently, $0\leq U(p_O)<U(p_I)=0$, which is a contradiction.

On the other hand, by noting $U''(p)=\lambda_2(\lambda_2-1)Bp^{\lambda_2-2}-\lambda_1(\lambda_1-1)Ap^{\lambda_1-2}$, the equation $U''(p)=0$ has at most one solution in $(p_O,+\infty)$.

Therefore, $U''(p_I)>0$, and then
\[
\begin{split}
&r(G(p_I)-k_1p_I-k_0)-\mu p_I(G'(p_I)-k_1)-\frac{1}{2}\sigma^2p_I^2G''(p_I)\\
&\quad>rH(p_I)-\mu p_IH'(p_I)-\frac{1}{2}\sigma^2p_I^2H_{-}''(p_I)=0,
\end{split}
\]
i.e.,
\[
p_I>\exp(-\mu\delta)(C+rK_I),
\]
where we have used $rG(p_I)-\mu p_IG'(p_I)-\sigma^2p_I^2G''(p_I)/2=p_I-C$.

\noindent 3. The inequality $p_I^{(1)}<\exp(-\mu\delta)(C-rK_O)$ follows from $p_I^{(1)}\leq p_O$ and Corollary \ref{c:ExTrUpBo}. The proof of the inequality $p_I^{(2)}>\exp(-\mu\delta)(C+rK_I)$ is similar to Step 2.
\end{proof}

\begin{lemma}\label{l:LargSolu}
Assume that $r>\mu$, $C-rK_O>0$, $C+rK_I> 0$ and $K_I+K_O\geq0$. Then the equation
\[
A(\lambda_2-\lambda_1)p^{\lambda_1}+\frac{\exp((\mu-r)\delta)}{r-\mu}(\lambda_2-1)p-\lambda_2\exp(-r\delta)\left(\frac{C}{r}+K_I\right)=0
\]
has only two solutions $p_1$ and $p_2$ in $(0, +\infty)$ satisfying $p_1\leq p_O$ and $p_2> p_O$. Furthermore, $\lambda_1Ap_2^{\lambda_1}+\exp((\mu-r)\delta)/(r-\mu)>0$.
\end{lemma}
\begin{proof}
The proof is similar to that of \cite[Lemma 5.5]{Zhang2015}.

\noindent 1. Defined a function $E$ by
\[
E(p):=A(\lambda_2-\lambda_1)p^{\lambda_1}+\frac{\exp((\mu-r)\delta)}{r-\mu}(\lambda_2-1)p-\lambda_2\exp(-r\delta)\left(\frac{C}{r}+K_I\right).
\]

Suppose that the equation $E(p)=0$ has three solutions in $(0, +\infty)$. Then by Rolle's mean value theorem, there is a positive number $\xi$ such that
$E''(\xi)=0$, i.e.,
\[
A(\lambda_2-\lambda_1)\lambda_1(\lambda_1-1)\xi^{\lambda_1-2}=0,
\]
which is impossible. Thus the equation $E(p)=0$ has at most two solutions in $(0, +\infty)$.

\noindent 2. In this step, we will estimate $E(p_O)$ and $E'(p_O)$.

We first estimate $E(p_O)$ as follows.
\[
\begin{split}
E(p_O)&=(\lambda_2-\lambda_1)\left(\frac{\exp((\mu-r)\delta)}{\mu-r}p_O+\exp(-r\delta)\left(\frac{C}{r}-K_O\right)\right)\\
&\qquad+\frac{\exp((\mu-r)\delta)}{r-\mu}(\lambda_2-1)p_O-\lambda_2\exp(-r\delta)\left(\frac{C}{r}+K_I\right)\\
&=-\lambda_2\exp(-r\delta)(K_I+K_O)\leq 0,
\end{split}
\]
where we have used continuity of the function $G$ at $p_O$ for the first equality.

Now we estimate $E'(p_O)$.
\[
\begin{split}
E'(p_O)&=(\lambda_1-\lambda_2)\frac{\exp((\mu-r)\delta)}{r-\mu}+(\lambda_2-1)\frac{\exp((\mu-r)\delta)}{r-\mu}\\
&=(\lambda_1-1)\frac{\exp((\mu-r)\delta)}{r-\mu}<0,
\end{split}
\]
where we have used $C^1$ continuity of the function $G$ at $p_O$ for the first equality.

\noindent 3. Note that $\lim\limits_{p\rightarrow 0^+}E(p)=+\infty$, $\lim\limits_{p\rightarrow +\infty}E(p)=+\infty$, $E(p_O)\leq 0$ and $E'(p_O)<0$. We find that the equation $E(p)=0$ has only two solutions $p_1$ and $p_2$ in $(0, +\infty)$ satisfying $p_1\leq p_O$ and $p_2> p_O$.

Furthermore, $E'(p_2)>0$. It follows that
\[
\begin{split}
&\lambda_1Ap_2^{\lambda_1}+\frac{\exp((\mu-r)\delta)}{r-\mu}\\
&\quad >\lambda_1Ap_2^{\lambda_1}+\frac{\lambda_2-1}{\lambda_2-\lambda_1}\cdot\frac{\exp((\mu-r)\delta)}{r-\mu}\\
&\quad =\frac{1}{\lambda_2-\lambda_1}E'(p_2)>0.
\end{split}
\]
The proof is complete.
\end{proof}

Recall the problem \eqref{e:optimalD}. The following Theorem \ref{t:optimalIOD} provides a solution of entry and exit decisions and an explicit expression of the maximal expected present value from the project. To prove Theorem \ref{t:optimalIOD}, we need the following lemma.

\begin{lemma}[{\cite[Lemma 4.1]{Zhang2015}}]\label{l:CloUni}
Assume that $r>\mu$. Then the family of random variables $\{\exp((\mu-r-\sigma^2/2)\tau+\sigma B(\tau)): \tau\in\mathcal{T}\}$ is uniformly integrable, where $\mathcal{T}$ is the collection of all stopping times.
\end{lemma}

\begin{theorem}\label{t:optimalIOD}
In each case of Theorem \ref{t:EntrDeci}, $\tau_I^*$ is an optimal entry decision time, and $\tau_O^*$ is an optimal exit decision time, where $\tau_O^*:=+\infty$ if $C\leq rK_O$ and $\tau_O^*:=\inf\{t:t>\tau_I^*\ \text{and}\ P(t)\leq p_O\}$ if $C>rK_O$, respectively. In addition, we have $J(p)=H(p)$.
\end{theorem}
\begin{proof}
1. Recall the problem \eqref{e:optimalIns}
\[
\begin{split}
\widetilde{J}(p):=\sup\limits_{\tau_I\leq\tau_O}\mathbb{E}^p&\left[\int_{\tau_I}^{\tau_O}\exp(-r t)(P(t)-C)\mathrm{d}t\right.\\
&\quad\left.-\exp(-r\tau_I)(k_1P(\tau_I)+k_0)-\exp(-r \tau_O)(l_1P(\tau_O)+l_0)\right].
\end{split}
\]

\noindent 2. Define a sequence of stopping times $(R_k, k\in\mathbb{N})$ by $R_k:=\inf\{t: t>0, P(t)>k\}$. For an entry decision time $\tau_I$, define a sequence of stopping times $(S_k, k\in\mathbb{N})$ by $S_k:=\tau_I\wedge k\wedge R_k$. Similarly, for an exit decision time $\tau_O(\geq \tau_I)$, define a sequence of stopping times $(T_k, k\in\mathbb{N})$ by $T_k:=\tau_O\wedge k\wedge R_k$. Define an operator $\mathcal{A}$ by
\[
\mathcal{A}\varphi(p):=\frac{1}{2}\sigma^2p^2\varphi''(p)+\mu p \varphi'(p)-r \varphi(p).
\]
Note that the functions $G$ and $H$ are $C^1$ convex functions. It follows from Meyer-It\^{o} formula (see\cite[p.~218, Theorem 70]{Pro2005}) that
\[
\begin{aligned}
\exp(-rT_k)G(P(T_k))=&\exp(-rS_k)G(P(S_k))+\int_{S_k}^{T_k}\exp(-rt)\mathcal{A}G(P(t))\mathrm{d}t\\
&+\sigma\int_{S_k}^{T_k}\exp(-rt)P(t)G'(P(t))\mathrm{d}B(t)
\end{aligned}
\]
and
\[
\begin{aligned}
\exp(-rS_k)H(P(S_k))=&H(P(0))+\int_{0}^{S_k}\exp(-rt)\mathcal{A}H(P(t))\mathrm{d}t\\
&+\sigma\int_{0}^{S_k}\exp(-rt)P(t)H'(P(t))\mathrm{d}B(t).
\end{aligned}
\]
Therefore,
\begin{equation}\label{e:CutVal3}
\begin{aligned}
&\int_{S_k}^{T_k}\exp(-rt)\left(P(t)-C\right)\mathrm{d}t-\exp(-r S_k)(k_1P(S_k)+k_0)\\
&\quad\hskip 17em -\exp(-r T_k)(l_1P(T_k)+l_0)\\
&\quad=\int_{S_k}^{T_k}\exp(-rs)\left(\mathcal{A}G(P(t))+P(t)-C\right)\mathrm{d}t
+\int_{0}^{S_k}\exp(-rt)\mathcal{A}H(P(t))\mathrm{d}t\\
&\quad\quad-\exp(-r S_k)(k_1P(S_k)+k_0-G(P(S_k))+H(P(S_k)))\\
&\quad\quad-\exp(-r T_k)(l_1P(T_k)+l_0+G(P(T_k)))\\
&\quad\quad+\sigma\int_{S_k}^{T_k}\exp(-rt)P(t)G'(P(t))\mathrm{d}B(t)\\
&\quad\quad+\sigma\int_{0}^{S_k}\exp(-rt)P(t)H'(P(t))\mathrm{d}B(t)+H(P(0)).
\end{aligned}
\end{equation}

\noindent 3.
By Doob's optional sampling theorem (see \cite[p.~19, Theorem 3.22]{KarShr1991} or \cite[p.~9, Theorem 16]{Pro2005}), we have
\[
\mathbb{E}^p\left[\int_{S_k}^{T_k}\exp(-rs)P(t)G'(P(t))\mathrm{d}B(t)\right]=0
\]
and
\[
\mathbb{E}^p\left[\int_{0}^{S_k}\exp(-rs)P(t)H'(P(t))\mathrm{d}B(t)\right]=0.
\]
In addition, note that, by the proof of Theorem \ref{t:ExitTrigger} and \ref{t:EntrDeci},
\[
\mathcal{A}G(p)+p-C\leq 0,
\]
\[
l_1p+l_0+G(p)\geq 0,
\]
\[
\mathcal{A}H(p)\leq 0,
\]
\[
k_1p+k_0-G(p)+H(p)\geq 0.
\]
Then it follows from (\ref{e:CutVal3}) that
\begin{equation}\label{e:CutLess}
\begin{aligned}
\mathbb{E}^p\left[\int_{S_k}^{T_k}\exp(-rt)\left(P(t)-C)\right)\mathrm{d}t-\exp(-r S_k)(k_1P(S_k)+k_0)\right.&\\
\left.-\exp(-r T_k)(l_1P(T_k)+l_0)\right]&\leq H(p).
\end{aligned}
\end{equation}

Note that
\begin{equation}\label{e:rgeqmu}
\begin{split}
&\lim\limits_{{}^{T\rightarrow+\infty}}\mathbb{E}\left[\left.\int_{0}^{T}\exp(-r t)P(t)\mathrm{d}t\right|P(0)=p\right]\\
&\quad=\lim\limits_{{}^{T\rightarrow+\infty}}\int_0^{T}\exp(-rt)p\exp(\mu t)\mathrm{d}t\\
&\quad=\frac{p}{r-\mu},
\end{split}
\end{equation}
where we have used the fact that the process
\[
\left(\exp\left(-\frac{1}{2}\sigma^2 t+\sigma B(t)\right), t\geq 0\right)
\] is a martingale (see \cite[p.~288, Corollary 5.2.2]{App2009}) for the first equality.

Taking limits in the inequality \eqref{e:CutLess} and using Lemma \ref{l:CloUni} and (\ref{e:rgeqmu}), we get
\begin{equation}\label{e:LimCutVal3}
\begin{aligned}
\mathbb{E}^p\left[\int_{\tau_I}^{\tau_O}\exp(-rt)\left(P(t)-C)\right)\mathrm{d}t-\exp(-r \tau_I)(k_1P(\tau_I)+k_0)\right.&\\
\left.-\exp(-r \tau_O)(l_1P(\tau_O)+l_0)\right]&\leq H(p).
\end{aligned}
\end{equation}
Thus
\begin{equation}\label{e:CutValleq}
\widetilde{J}(p)\leq H(p).
\end{equation}

\noindent 4. Take $\tau_I=\tau_I^*$ in the definition of $S_k$ and $\tau_O=\tau_O^*$ in the definition of $T_k$, where $\tau_I^*$'s are defined in Theorem \ref{t:EntrDeci}, and $\tau_O^*$'s are given by $\tau_O^*:=+\infty$ if $C\leq rK_O$ and $\tau_O^*:=\inf\{t:t>\tau_I^*\ {\rm and}\ P(t)\leq p_O\}$ if $C>rK_O$, respectively.

Note that
\[
\mathcal{A}H(P(t))=0\;\mathrm{for}\;0<t<\tau^*_I\;\text{by the proof of Theorem \ref{t:EntrDeci}},
\]
and
\[
\mathcal{A}G(P(t))+P(t)-C=0\;\mathrm{for}\;\tau^*_I<t<\tau^*_O\;\text{by the proof of Theorem \ref{t:ExitTrigger}}.
\]

Then, in light of (\ref{e:CutVal3}), we have
\[
\begin{aligned}
&\mathbb{E}^p\left[\int_{S_k}^{T_k}\exp(-rt)\left(P(t)-C)\right)\mathrm{d}t-\exp(-r S_k)(k_1P(S_k)+k_0)\right.\\
&\qquad\left.-\exp(-r T_k)(l_1P(T_k)+l_0)\right]\\
&\quad=-\mathbb{E}^p\left[\exp(-r S_k)(k_1P(S_k)+k_0-G(P(S_k))+H(P(S_k)))\right.\\
&\quad\qquad\quad\left.+\exp(-r T_k)(l_1P(T_k)+l_0+G(P(T_k)))\right]+H(p).
\end{aligned}
\]

Since the functions $G$ and $H$ are at most linear growth, we find, through a similar way to that of (\ref{e:LimCutVal3}),
\[
\begin{aligned}
&\mathbb{E}^p\left[\int_{\tau_I^*}^{\tau_O^*}\exp(-rt)\left(P(t)-C)\right)\mathrm{d}t-\exp(-r \tau_I^*)(k_1P(\tau_I^*)+k_0)\right.\\
&\qquad\left.-\exp(-r \tau_O^*)(l_1P(\tau_O^*)+l_0)\right]\\
&\quad=H(p).
\end{aligned}
\]

The above equality and inequality \eqref{e:CutValleq} show us that $\widetilde{J}(p)=H(p)$ and $(\tau_I^*,\tau_O^*)$ is a solution of the problem \eqref{e:optimalIns}. Consequently, These and Theorem \ref{t:TranSpec} complete the proof.
\end{proof}
\begin{example}
Take $r=0.2$, $\mu=0.1$, $\sigma=0.3$, $\delta=1$, $C=10$, $K_I=-20$ and $K_O=10$. Then we have $r>\mu$, $C-rK_O>0$, $C+rK_I>0$, $K_I+K_O<0$ and $p_O=2.66841<5.42902=\exp(-\mu\delta)(C+rK_I)$. Thus we apply \eqref{i:DoubEntrTrig} of Theorem \ref{t:EntrDeci}, and get $p_I^{(1)}=1.96101$ and $p_I^{(2)}=6.94641$. Therefore, by Theorem \ref{t:optimalIOD}, $\tau_I^*:=\inf\{t:t>0, P(t)\leq1.96101\ {\rm or}\ P(t)\geq6.94641\}$ is an optimal entry decision time and $\inf\{t:t>\tau_I^*, P(t)\leq2.66841\}$ is an optimal exit decision time.
\end{example}

\bibliographystyle{mybst2}
\bibliography{Xbib}

\end{document}